\documentclass[12pt,twoside]{article}
	\usepackage[utf8]{inputenc}
	\usepackage{a4wide}
	\usepackage{amsfonts}
	\usepackage{amsmath}
	\usepackage{amssymb}
	\usepackage{amsthm}
	\usepackage{mathrsfs}
	\usepackage{mathtools}
	\usepackage{enumerate}
	\usepackage[tracking=smallcaps]{microtype}

	\newcommand{\F}{{\sf F}}
	\newcommand{\HH}{{\sf H}}
	\newcommand{\Pp}{\mathbb{P}}

	\newcommand{\Oo}{\mathscr{O}}
	\newcommand{\Ff}{\mathscr{F}}
	\newcommand{\Gg}{\mathscr{G}}

	\newcommand{\Ee}{\mathcal{E}}	
	\newcommand{\Ll}{\mathscr{L}}	
	\newcommand{\Zz}{\mathbb{Z}}	
	\newcommand{\Ss}{{\sf S}}

	\DeclareMathOperator{\Coker}{Coker}
	\DeclareMathOperator{\chara}{char}
	\DeclareMathOperator{\Hom}{Hom}
	\DeclareMathOperator{\Ext}{Ext}

	\newcommand{\isom}{\simeq}
	\newcommand{\xto}{\xrightarrow}
	\newcommand{\dual}{\vee}
	\renewcommand{\phi}{\varphi}
	\renewcommand{\epsilon}{\varepsilon}
	\renewcommand{\tilde}{\widetilde}

	\newcommand{\Ot}{\bigoplus_{t\in\mathbb{Z}}}
	\newcommand{\St}{\sum_{t\in\mathbb{Z}}}
	\newcommand{\Omg}{{\Omega^1_{\Pp^N}}}
	\newcommand{\MM}{{2^{\lfloor n/2 \rfloor+1}}}
	\newcommand{\sizephin}{{2^{\lfloor (n+1)/2\rfloor}}}
	\newcommand{\sizePhin}{\MM} 
	\newcommand{\ghostlf}{\hspace{1cm}}
	\newcommand{\dn}{\frac{1}{2}n(p-1)}

	\newcommand{\pscirclebox}[1]{\subsection*{Step #1}}

	\theoremstyle{plain}
	\newtheorem{lemma}{Lemma}[section]
	\newtheorem*{langerlemma}{Langer's Lemma}
	\newtheorem{corollary}[lemma]{Corollary}
	\newtheorem{proposition}[lemma]{Proposition}
	\newtheorem*{proposition*}{Proposition}
	\newtheorem*{corollary*}{Corollary}
	\newtheorem{theorem}{Theorem}
	\newtheorem*{theorem*}{Theorem}

	\theoremstyle{definition}
	\newtheorem{definition}[lemma]{Definition}

	\theoremstyle{remark}
	\newtheorem*{remark}{Remark}

	\author{Piotr Achinger}
	\title{Frobenius Push-Forwards on Quadrics}
	\numberwithin{equation}{section}
\begin{document}

\maketitle

\begin{abstract}
We generalize, explain and simplify
Langer’s results concerning Frobenius direct images of line bundles 
on quadrics, describing explicitly the decompositions of higher Frobenius
push-forwards of arithmetically Cohen-Macaulay bundles into 
indecomposables, with an
additional emphasis on the case of characteristic two.
These results are applied to check which Frobenius push-forwards of
the structure sheaf are tilting. 
\end{abstract}

\section*{Introduction}

In \cite{AL}, A. Langer computed the Frobenius push-forwards of 
line bundles on quadrics. 
However, the computations worked only for odd
characteristic and explicit formulas for the push-forward were given
only for the first Frobenius direct image. In this paper, we determine
the push-forwards of line and spinor bundles on smooth quadrics in 
arbitrary positive characteristic. But mostly, we explain and simplify
the aforementioned paper, reproving nearly all of the statements.

To illustrate our method, we briefly show how it can be used to determine
Frobenius push-forwards of line bundles on a projective space $\Pp^N$ 
(this method is used in \cite{R}, Lemma 2.1). If the absolute Frobenius 
morphism on $\Pp^N$ is denoted by $\F$, its $s$-th composition by $\F^s$, 
the push-forward in question can be written as
\[ 
	\F^s_* (\Oo(a)) = \Ot \Oo(t)^{\alpha^s(t, a)} 
\]
for some integers $\alpha^s(t, a)$ (the existence of such a decomposition
follows directly from Horrocks' splitting criterion and the projection 
formula). 

To compute $\alpha^s(t, a)$, let us write the projection formula
using the bundle $\Omg(-b)$:
\[
	\F^s_*(\F^{s*} \Omg(a-bq)) = \F^s_*(\Oo(a))\otimes \Omg(-b), 
\]
so comparing dimensions of the cohomology groups we get
\[ 
	h^1(\F^{s*} \Omg(a-bq)) 
	  = \St \alpha^s(t,a) \cdot h^1(\Omg(t-b)). 
\]
But $h^1(\Omg(t-b))=\delta_{t,b}$, so the right hand side is just 
$\alpha^s(b,a)$.

On the other hand, the dimension of $\HH^1(\Pp^N, \F^{s*}\Omg(t))$ can be
computed as
\[ 
	\dim\left( 
	  \underbrace{
	    k[x_0,\ldots, x_N] / (x_0^q, \ldots, x_N^q)
	  }_{D^{(s)}}
	\right)_{\!\!\!\displaystyle t} 
	= \sum_{j=0}^{N+1} (-1)^j \binom{N+1}{j} \binom{N+t-jq}{N}
\]
(see Lemma \ref{lemma:bm}). Hence we obtain
\[ 
	\alpha^s(t, a) 
	  = \dim D^{(s)}_{a-tq} 
	  = \sum_{j=0}^{N+1} (-1)^j \binom{N+1}{j} \binom{N+a-tq-jq}{N}. 
\]

On quadrics, the situation is quite similar. It is well known 
that any ACM (\emph{arithmetically Cohen-Macaulay}, i.e., with vanishing 
$h^i(\Ee(t))$ for $0 < i < n$) bundle on a smooth $n$-dimensional quadric 
decomposes into a direct sum of line bundles and twisted spinor 
bundles. We use the above method to compute the coefficients in this 
decomposition. The result (see Theorem \ref{theorem:fofs}) is that
\[ 
	\F^s_* (\Oo(a)) 
	  = \Ot \Oo(t)^{\beta^s(t, a)} \oplus \Ot \Ss(t)^{\gamma^s(t, a)}, 
\]
where $\Ss$ is the spinor bundle or the sum of the two half-spin bundles 
on $Q_n$ (see Section \ref{section:spinor}) and the coefficients $\beta$ and
$\gamma$ are given by the formulas
\begin{align*}
	\beta^s(t, a) &= \dim C^{(s)}_{a-tq} \, , \\
	\gamma^s(t, a) &= \frac{1}{\MM} \dim M^{(s)}_{a-(t-1)q} \, , 
\end{align*}
where $C^{(s)}$, and $M^{(s)}$ are certain graded modules defined in Section 
\ref{section:algebras}. The decomposition of $\F^s_* (\Ss(a))$ is also given.
This description allows us to give explicit 
vanishing criteria for these coefficients (Theorems \ref{theorem:van} and 
\ref{theorem:van2}), from which we easily derive corollaries concerning 
the push-forwards being tilting (Theorem \ref{theorem:tilting}). The last
section of the paper contains a comment on possible extension of these
results to singular quadrics.

In particular, for $p=2$ the formulas become easier and we can be a little bit
more explicit. We extend the main theorems of \cite{AL} to this
case. 

The paper \cite{AL} was inspired by Samokhin's paper \cite{S1}.
Frobenius direct images of the structure sheaf are of particular interest
because they can produce tilting bundles and allow us to study 
$\mathscr{D}$-affinity in positive characteristic
(\cite{S1}, \cite{AL}, \cite{S2}).

\medskip
{\bf Acknowledgements.} I would like to thank Prof. Adrian Langer for
giving me the idea for writing this paper and for many helpful clues. 

\tableofcontents

\section{Preliminaries} 
\label{section:pre}


\subsection[The Frobenius morphism and some projection formulas]
           {The Frobenius morphism and some projection formulas}
\label{section:frobenius}

Let $X$ be a projective variety over an algebraically closed field $k$ of 
characteristic $p>0$. The absolute Frobenius morphism $\F: X\to X$ is the 
mapping of schemes acting as identity on the underlying topological space 
and as the $p$-th power map on the structure sheaf. It is not a map of 
$k$-schemes. Denote by $\F^s$ the $s$-th composition of the Frobenius
morphism and set $q = p^s$ once and for all.

Let $\Ff$ be a locally free sheaf of rank $r$ on $X$. If $X$ is smooth 
then $\F$ is flat and the sheaf $\F^s_* \Ff$ is also locally free, of 
rank $rq^{\dim X}$. The sheaf $\F^{s*} \Ff$ is locally free of rank $r$,
and it is glued as a bundle using the cocycle obtained by raising the
coefficients of the transition matrices defining $\Ff$ to the $q$-th power.
If $\Ff$ is a line bundle, we infer from the above description of its 
pull-back that $\F^{s*} \Ff \isom \Ff^{\otimes q}$.

Let $\Gg$ be a locally free sheaf. Since the Frobenius is an
affine morphism, so that $\HH^i(X, \Ff) = \HH^i(X, \F_* \Ff)$, we immediately
deduce from the projection formula
$\F^s_* ( \Ff \otimes \F^{s*} \Gg ) \isom \F^s_* (\Ff) \otimes \Gg$
the following formulas concerning cohomology:
\begin{align}
\label{eq:proj1} 
	\HH^i(X, \Ff \otimes \F^{s*} \Gg) 
	  & \isom \HH^i(X, (\F^s_*\Ff) \otimes \Gg), \\
\label{eq:proj2} 
	\HH^i(X, \Ff (tq)) 
	  &\isom \HH^i(X, (\F^s_* \Ff)(t)), \\
\label{eq:proj3} 
	\HH^i(X, (\F^{s*} \Gg)(a+tq)) 
	  &\isom \HH^i(X, \F_*(\Oo(a)) \otimes \Gg(t)). 
\end{align}

\begin{remark} These isomorphisms are not $k$-linear, but
the dimensions over $k$ on both sides agree.
\end{remark}

\begin{definition}
	A coherent sheaf $\Ff$ on a projective variety $X$ with a very ample
	line bundle $\Ll$ is called \emph{arithmetically Cohen-Macaulay} 
	(\emph{ACM}) if
	\[ 
		\Ot \HH^i(X, \Ff \otimes \Ll^{\otimes t}) = 0 
	\quad \text{for} \quad 0 < i < \dim X . \]
\end{definition}

\noindent 
Formula (\ref{eq:proj2}) shows that the Frobenius push-forward of any coherent
ACM sheaf is ACM.


\subsection{Quadrics}
\label{section:quadrics}

Let $n$ be a positive integer. The \emph{smooth $n$-dimensional quadric} $Q_n$
(or simply $Q$) is the hypersurface in $\Pp^N$, $N=n+1$ defined by the 
equation $Q_n = 0$ where
\[ 
	Q_n = x_0^2 + x_1 x_2 + \ldots + x_n x_{n+1} 
\]
if $n$ is odd and
\[ 
	Q_n = x_0 x_1 + \ldots + x_n x_{n+1} 
\]
if $n$ is even. If $\chara\, k \neq 2$ then we can take a linear change of 
coordinates on $\Pp^N$ such that the quadric $Q_n$ is given by the simpler
equation $x_0^2 + \ldots + x_N^2 = 0$.

For completeness, let us also state here that by the adjunction formula $Q_n$ 
is a Fano variety with the canonical bundle $\omega_X = \Oo_Q(-n)$ and Hilbert
polynomial $q_t := \chi(\Oo_Q(t))$ equal to
\[ 
	q_t = \binom{N+t}{N} - \binom{N+t-2}{N}. 
\]

\begin{remark}
To simplify the calculations, we will assume that $n>2$. This is not a real 
restriction since $Q_1 \isom \Pp^1$ ($Q_1$ being the image of the Veronese 
embedding of $\Pp^1$ in $\Pp^2$) and $Q_2 \isom \Pp^1\times \Pp^1$ ($Q_2$ 
being the image of the Segre embedding of $\Pp^1\times\Pp^1$ in $\Pp^3$)
and everything we would want to say in these cases could be easily derived
from what has been said in the example in the Introduction.
\end{remark}


\subsection{Spinor bundles}
\label{section:spinor}

Now we shall recall the basic facts about the so-called \emph{spinor bundles}
on smooth quadrics.
On $Q_n$, we have a single spinor bundle $\Sigma$ if $n$ is odd and two spinor
bundles $\Sigma_+$, $\Sigma_-$ (sometimes called half-spin) if $n$ is even. 
There are many equivalent ways of introducing them present in the literature.
We shall use the following:

{\bf Matrix factorizations.} A \emph{matrix factorization} of a
	  polynomial $f$ with $f(0, \ldots, 0) = 0$ is a pair $(\phi, \psi)$ of 
	  square matrices of the same size such that $\phi\cdot\psi = f\cdot id 
	  = \psi\cdot\phi$. It was first observed by Eisenbud in \cite{EMF} that
	  given an appropriate notion of a morphism, the matrix factorizations 
	  of $f$ form a category that is equivalent to the stable category of 
	  maximal Cohen-Macaulay modules over the local ring $\Oo_{k^n, 0}/(f)$ 
	  of the hypersurface defined by $f = 0$. The module corresponding to
	  $(\phi, \psi)$ is $\Coker\,\phi$ where $\phi$ is regarded as a map
	  $\Oo^m\to \Oo^m$, $m$ being the size of both matrices; it is an 
	  $\Oo/(f)$--module. 

	  Using this technique,	Eisenbud, Buchweitz and Herzog in \cite{EBH} 
	  then classified all indecomposable \emph{graded} maximal 
	  Cohen-Macaulay modules over $k[x_0, \ldots, x_N]/(Q_n)$. Their 
	  description remains valid over any field $k$. It turns out that 
	  apart from the free MCMs, there is (up to shift) only one 
	  indecomposable module $M$ if $n$ is odd and there are two of them, 
	  $M_+$ and $M_-$ if $n$ is even. The corresponding matrix 
	  factorizations can be defined	inductively as follows 
	  (see \cite{AL}, Section 2.2): 
	  \[ 
		\phi_{-1} = (x_{0}) = \psi_{-1}, 
		\quad \phi_0 = (x_0), 
		\quad \psi_0 = (x_1), 
	  \]
	  \[ 
		\phi_n = \left( \begin{array}{cc}
			\phi_{n-2}      & x_n \cdot id  \\
			x_{n+1}\cdot id & -\psi_{n-2}   \\
		\end{array} \right), 
		\quad
		\psi_n = \left( \begin{array}{cc}
			\psi_{n-2}      & x_n \cdot id  \\
			x_{n+1}\cdot id & -\phi_{n-2}   \\
		\end{array} \right).
	  \]

	  To define the spinor bundles using these matrix factorizations, we
	  consider $\phi_n$ and $\psi_n$ as maps between free sheaves on 
	  $\Pp^N$, i.e., $\phi_n, \psi_n: \Oo_{\Pp^N}(-2)^\sizephin \to 
	  \Oo_{\Pp^N}(-1)^\sizephin$. Then for odd $n$ we can define $\Sigma$ to be 
	  the cokernel of $\phi_n=\psi_n$, which is supported on $Q_n$.
	  For even $n$ we define $\Sigma_+$ to be the cokernel of $\phi_n$ and
	  $\Sigma_-$ to be the cokernel of $\psi_n$.
	  Additional references: \cite{Yoshino},  \cite{K} and \cite{A}.
	
As mentioned above, we have the following exact sequences of sheaves on 
$\Pp^N$:
\[
	0 \to \Oo_{\Pp^N}(-2)^\sizephin 
	  \xto{\phi_n=\psi_n} \Oo_{\Pp^N}(-1)^\sizephin \to i_* \Sigma \to 0 
\]
if $n$ is odd and
\[ 
	0 \to \Oo_{\Pp^N}(-2)^\sizephin 
	  \xto{\phi_n} \Oo_{\Pp^N}(-1)^\sizephin \to i_* \Sigma_+ \to 0, 
\]
\[ 
	0 \to \Oo_{\Pp^N}(-2)^\sizephin 
	  \xto{\psi_n} \Oo_{\Pp^N}(-1)^\sizephin \to i_* \Sigma_- \to 0 
\]
if $n$ is even. It follows that the spinor bundles are arithmetically 
Cohen-Macaulay. In fact, as implied by the Eisenbud-Buchweitz-Herzog theorem,
they provide a full description of ACM bundles on $Q_n$:
\begin{theorem*} Any coherent ACM sheaf $\Ff$ on a smooth quadric $Q_n$ is 
a direct sum of line bundles and twisted spinor bundles.
\end{theorem*}

In what follows, we shall use the bundle $\Ss$ defined by $\Ss=\Sigma$ for $n$ 
odd and $\Ss=\Sigma_+\oplus\Sigma_-$ for $n$ even. We thus have the exact 
sequence of sheaves on $\Pp^N$:
\begin{equation}\label{eq:exacts}
	0 \to \Oo_{\Pp^N}(-2)^\sizePhin 
	  \xto{\Phi} \Oo_{\Pp^N}(-1)^\sizePhin \to i_* \Ss \to 0, 
\end{equation}
where $(\Phi_n, \Psi_n)$ is the matrix factorization defined by 
$\Phi_n = \phi_n$, $\Psi_n = \psi_n$ if $n$ is odd and $\Phi_n = \phi_n 
\oplus \psi_n$, $\Psi_n = \psi_n \oplus \phi_n$ if $n$ is even. The exact 
sequence (\ref{eq:exacts}) allows us to compute the Hilbert polynomial 
$s_t := \chi(\Ss(t))$ of $\Ss$:
\[ 
	s_t = \MM \binom{n + t - 1}{n}. 
\]

\section{Some graded algebras and modules}
\label{section:algebras}

As we shall see in Section \ref{section:differentials}, the Euler sequence
allows us to translate dimensions of sheaf cohomology groups into dimensions
of gradings of certain 0-dimensional graded modules. In this section we
develop technical results which let us accomplish the tasks in Section 
\ref{section:fofs}.


\subsection{Definitions}
\label{section:defs}

Let $Q$ be the equation of the $n$-dimensional quadric as in Section 
\ref{section:quadrics}. Recall that $q=p^s$ and $N=n+1$. We set
\begin{align*}
	S	&= k[x_0, \ldots, x_N], \\
	R	&= S/(Q), \\
	A^{(s)} &= R/(x_0^q + x_1^q, x_2^q, \ldots, x_N^q), \\
	B^{(s)} &= A/(x_0^q) = R/(x_0^q, x_1^q, \ldots, x_N^q), \\
	C^{(s)} &= A/(0:x_0^q),  \\
	D^{(s)} &= S/(x_0^q, \ldots, x_N^q), \\
	M^{(s)} &= (0:_{A^{(s)}} x_0^q)A^{(s)}/ (x_0^q)A^{(s)}. 
\end{align*}

\begin{remark}
 The strange generator $x_0^q + x_1^q$ in the 
definition of 
$A^{(s)}$ is used to make $A^{(s)}$ zero-dimensional (or to ensure that 
$(x_0^q+x_1^q, x_2^q, \ldots, x_N^q)$ is an R-regular sequence). It is easy to
check that the ring $S/(Q, x_1^q, \ldots, x_N^q)$ is one-dimensional when $n$ is even, 
i.e., $Q=x_0 x_1 + x_2 x_3 + \ldots$ and $p=2$. This is due to the fact that
$x_0^2$ does not appear in $Q$. In any other 
case, we can assume that $A=S/(Q, x_1^q, \ldots, x_N^q)$ as in \cite{AL}.
\end{remark}

%

By Section \ref{section:spinor}, we can write the module $\Gamma_*(\Ss)$ as the
cokernel of a map $\Phi: S[-2]^\MM\to S[-1]^\MM$ ($\Phi$ is an $\MM\times\MM$ 
matrix of linear forms). The following definitions pertain to spinor bundles 
and will be needed in Section \ref{section:fofs}: 
\begin{align*}
	Z              &= \Gamma_*(\Ss) = \Coker (\Phi), \\
	\tilde A^{(s)} &= Z/(x_0^q + x_1^q, x_2^q, \ldots, x_N^q)Z, \\
	\tilde B^{(s)}        &= Z/(x_0^q, x_1^q, \ldots, x_N^q)Z, \\
 	\tilde C^{(s)} &= \tilde A/(0:x_0^q),  \\
	\tilde M^{(s)} &= (0:_{\tilde A^{(s)}} x_0^q)\tilde A^{(s)}
	                  / (x_0^q)\tilde A^{(s)}. 
\end{align*}

\noindent
Recall that $Z$ is a maximal Cohen-Macaulay $R$-module and that 
$x_1, \ldots, x_N$ is a $Z$-regular sequence when $Z$ is considered as an 
$S$-module. Moreover, $\dim Z_d = \MM \binom{n + d - 1}{n} = s_d$. 


\subsection{Dividing MCMs by $q$-th powers} 
\label{section:koszul}

Recall that in the example in the Introduction, $\dim D^{(s)}_d$ is the number of 
monomials in $x_0, \ldots, x_N$ of degree $d$ with all exponents $<q$, so by
the inclusion-exclusion principle we obtain the combinatorial formula (which we
already used there):
\begin{equation} \label{eq:dsd}
	\dim D^{(s)}_d = \sum_{i=0}^{N+1} (-1)^j \binom{N+1}{j}\binom{N+d-jq}{N}. 
\end{equation}
In our study of spinor bundles, we shall need a more general statement. The 
following lemma explains this combinatorial formula in more algebraic terms.

\begin{lemma} \label{lemma:koszul}
Let $M$ be a graded module over a graded algebra $R$ generated by $R_1$ over
a field $k=R_0$. Let $(x_1, \ldots, x_k) \in R_q$ be a regular sequence on $M$ 
and $I=(x_1,\ldots,x_k)$. Then 
\[
	\dim_k \left( M/IM \right)_d 
	  = \sum_{j=0}^{k} (-1)^j \binom{k}{j} \dim_k M_{d-jq}. 
\] 
\end{lemma}

\begin{proof}
We construct the Koszul complex $C_* = M\otimes {\sf K}(x_1, \ldots, x_k)$. By 
\cite{M}, Theorem 43 (or \cite{E}, Corollary 17.5) we have $H_i(C_*) = 0$ for 
$i>0$ and $H_0(C_*)=M/IM$. Hence 
\[
	\dim_k (M/IM)_d = \sum_{i\geq 0} (-1)^j \dim_k (C_j)_{d-jq}, 
\]
since the maps in the Koszul complex have degree $q$. But 
$
	C_j = \Lambda^{N+1-j} R^{N+1} \otimes M \isom M^{\binom{k}{j}}  
$,
which finishes the proof.
\end{proof}

Note also that by \cite{E}, Corollary 17.8, if $(x_0, \ldots, x_N)$ is an
$M$-regular sequence then so is $(x_0^q, \ldots, x_N^q)$. We deduce (\ref{eq:dsd})
once again, together with
 \begin{align}
  \label{eq:asd}     
	\dim A^{(s)}_d        &= \sum_{j=0}^N (-1)^j \binom{N}{j} q_{d-jq}, \\
  \label{eq:atsd}     
	\dim\tilde A^{(s)}_d &= \sum_{j=0}^N (-1)^j \binom{N}{j} s_{d-jq}.
 \end{align}


\subsection{Dimensions of $B^{(s)}_d$ and $\tilde B^{(s)}_d$}
\label{section:dimensions}

We have the following two short exact sequences of graded modules:
\begin{align}
  \label{eq:exact1} 
    0 \to C^{(s)}[-q] \xto{x_0^q}  A^{(s)} \to B^{(s)} \to 0, \\
  \label{eq:exact2} 
    0\to \tilde C^{(s)}[-q] \xto{x_0^q} \tilde A^{(s)} \to \tilde B^{(s)} \to 0, 
\end{align}
Seeing that $\dim M^{(s)}_d = \dim B^{(s)}_d - \dim C^{(s)}_d$, we obtain 
$\dim B^{(s)}_{d} = \dim A^{(s)}_{d} + \dim M^{(s)}_{d-q} - \dim B^{(s)}_{d-q}$ 
(and the same with the tildes).
This gives the formulas
\begin{align}
  \label{eq:bsd}
     \dim B^{(s)}_{d} &= \sum_{j\geq 0} (-1)^j \dim A^{(s)}_{d-jq} 
                 + \sum_{j\geq 0} (-1)^j \dim M^{(s)}_{d-(j+1)q}\,, \\
  \label{eq:ksd}
     \dim \tilde B^{(s)}_{d} &= \sum_{j\geq 0} (-1)^j \dim \tilde A^{(s)}_{d-jq} 
                 + \sum_{j\geq 0} (-1)^j \dim \tilde M^{(s)}_{d-(j+1)q}\,. 
\end{align}

\section[The Frobenius morphism and the sheaf of differentials]
        {The Frobenius morphism\\ and the sheaf of differentials}
\label{section:differentials}

Now let us relate the commutative algebra from Section \ref{section:algebras} 
to cohomology groups to be used in Section \ref{section:fofs}.
The following standard result can be found e.g. in \cite{BM}.

\begin{lemma}\label{lemma:bm}
	Let $H\subseteq \Pp^N$ ($N>2$) be the hypersurface given by $f=0$. 
	Then there is an isomorphism
	of graded $S/(f)$-modules:
	\[ 
		\Ot \HH^1(H, (\F^{s*}(\Omega^1_{\Pp^N}|_H))(t)) \isom D^{(s)}/(f)
	\]
\end{lemma}

\noindent For $s=0$ we obtain
\begin{equation} \label{eq:bm2}
	h^1(\Omega^1_{\Pp^N}|_H (t)) = \delta_{t, 0} 
\end{equation}

When $Q$ is our quadric and $\Ss$ the spinor bundle defined in Section 
\ref{section:spinor}, we have

\begin{lemma} \label{lemma:bm2}
We have the following isomorphism of $R=S/(Q)$-modules:
	\[
	  \Ot \HH^1(\Ss\otimes \F^{s*}\Omega^1_{\Pp^N}|_Q(t)) \isom \tilde B^{(s)}. 
	\]
\end{lemma}

\begin{proof}
We write the exact sequence 
\[ 0 \to \Oo_{\Pp^N}(-2)^\MM \xto{\Phi} \Oo_{\Pp^N}(-1)^\MM \to \Ss \to 0, \]
tensor it by $\F^{s*}\Omega^1_{\Pp^N}(t)$ and look once again at the long
cohomology exact sequence
\begin{align*}
 \ldots\to &   \HH^1(\Pp^N, \F^{s*}\Omega^1_{\Pp^N}(t-2))^\MM \xto{\Phi} 
               \HH^1(\Pp^N, \F^{s*}\Omega^1_{\Pp^N}(t-1))^\MM \\
 \to       & \HH^1(Q, \Ss\otimes \F^{s*}\Omega^1_{\Pp^N}|_Q(t)) \to 
             \HH^2(\Pp^N, \F^{s*}\Omega^1_{\Pp^N}(t-2))^\MM. 
\end{align*}
The last group vanishes as $N>2$ (apply $\F^{s*}(-)\otimes\Oo_{\Pp^N}(d)$
to the Euler sequence and look at the cohomology exact sequence), so 
$\HH^1(Q, \Ss\otimes \F^{s*}\Omega^1_{\Pp^N}(t))$ is the cokernel of the map
\[
	\HH^1(\Pp^N, \F^{s*}\Omega^1_{\Pp^N}(t-2))^\MM \xto{\Phi} 
	  \HH^1(\Pp^N, \F^{s*}\Omega^1_{\Pp^N}(t-1))^\MM .
\]
Using our description of these groups from the previous lemma
we see that it is just the $t$-th grading of the graded module $\tilde B^{(s)}$.
\end{proof}

Clearly this lemma works (with the definitions slightly adjusted)
for arbitrary ACM sheaves over hypersurfaces (since
ACM sheaves are given by matrix factorizations).

As a corollary, for $s=0$ we have the following formula (see 
\cite{AL}, Proposition 4.1):
\begin{equation} \label{eq:bm3}
	h^1(\Ss\otimes \Omega^1_{\Pp^N}|_Q (t)) = \MM\cdot \delta_{t, 1} .
\end{equation}

\section{Decompositions of $\F^s_* (\Oo(a))$ and $\F^s_* (\Ss(a))$}
\label{section:fofs}

Let $\beta^s(t, a)$, $\gamma^s(t, a)$, $\delta^s(t, a)$ and $\varepsilon^s(t, a)$ 
be defined by the decompositions
\[ 
	\F^s_* (\Oo(a)) = \Ot\Oo(t)^{\beta^s(t,a)} 
	                  \oplus \Ot \Ss(t)^{\gamma^s(t,a)}, 
\]
\[ 
	\F^s_* (\Ss(a)) = \Ot\Oo(t)^{\delta^s(t,a)} 
		\oplus \Ot \Ss(t)^{\varepsilon^s(t,a)},
\]
where $\Ss$ is the spinor bundle or the sum of the two half-spin bundles as
defined in \ref{section:spinor}. 



\noindent \pscirclebox{1} 
By the projection formula ((\ref{eq:proj2}) for $\Ff = \Oo(a)$ or $\Ss(a)$
and $t=b$) we obtain
\begin{align} \label{eq:qabp}
    q_{a+bq} &= \St \beta^s(t,a)\cdot q_{t+b} 
             + \St \gamma^s(t,a)\cdot s_{t+b} ,\\
\label{eq:sabp}
    s_{a+bq} &= \St \delta^s(t,a)\cdot q_{t+b} 
             + \St \varepsilon^s(t,a)\cdot s_{t+b} .
\end{align} 
(the formulas hold for $b$ large enough, and hence for all $b$ since they are
equalities of polynomials in $b$).



\noindent \pscirclebox{2} 
Let $\psi = \Omega^1_{\Pp^N}|_{Q}$. By the projection formula ((\ref{eq:proj3})
for $\Gg = \psi$, $i=1$ and $t=-b$): 
\[
	\HH^1(Q_n, (\F^{s*}\psi)(a-bq)) 
	  = \HH^1(Q_n, \F^s_* (\Oo(a)) \otimes \psi(-b)). 
\]
By Lemma \ref{lemma:bm} we then have
\[
	\dim B^{(s)}_{a-bq} 
	= \dim\left(k[x_0, \ldots, x_N]/(Q, x_0^q, \ldots, x_N^q)\right)_{\! a-bq}
	= h^1(\F^s_*(\Oo(a)) \otimes \psi(-b)) 
\]
which can be rewritten as
\[ 
	\dim B^{(s)}_{a-bq} = \St \beta^s(t,a)\cdot h^1(\psi(t-b)) 
	                    + \St \gamma^s(t,a)\cdot h^1(\psi\otimes \Ss(t-b)). 
\]
But $h^1(\psi(t-b)) = \delta_{t,b}$ by (\ref{eq:bm2}) and 
$h^1(\psi\otimes \Ss(t-b)) = \MM\cdot \delta_{t, b+1}$ by (\ref{eq:bm3}), 
so this reduces to
$
	\dim B^{(s)}_{a-bq} = \beta^s(b, a) + \MM \cdot\gamma^s(b+1, a)
$
or 
\begin{align} \label{eq:bta}
	\beta^s(t, a) &= \dim B^{(s)}_{a-tq} - \MM\cdot \gamma^s(t+1, a).\\
\shortintertext{
Similarly, using Lemma \ref{lemma:bm2} and (\ref{eq:bm2}) one obtains
}
 \label{eq:dta}
	\delta^s(t, a) &= \dim \tilde B^{(s)}_{a-tq} - \MM\cdot \varepsilon^s(t+1, a). 
\end{align}



\noindent \pscirclebox{3} 
We put (\ref{eq:bta}) into (\ref{eq:qabp}), thus obtaining
\begin{align*}
  q_{a+bq} &= \St (\dim B^{(s)}_{a-tq} - \MM\cdot \gamma^s(t+1, a))q_{b+t} 
              + \St \gamma^s(t, a)\cdot s_{b+t}\\
           &= \St \dim (B^{(s)}_{a-tq})q_{t+b} 
              + \St \gamma^s(t, a)(s_{t+b} - \MM\cdot q_{t+b-1}) \\
           &= \St \dim (B^{(s)}_{a-tq})q_{t+b} 
              - \MM \St \gamma^s(t, a) \binom{n+t-2+b}{n}. 
\end{align*}
We rewrite this as
\begin{align} \label{eq:main}
	\St \dim(B^{(s)}_{a-tq})q_{b+t} - q_{a+bq}
	&= \MM \St \gamma^s(t+2, a)\binom{n+t+b}{n}. \\
\shortintertext{
Analogously, we get
}
\label{eq:mainsp}
	\St \dim (\tilde B^{(s)}_{a-tq})q_{b+t} - s_{a+bq}
	&= \MM \St \varepsilon^s(t+2, a)\binom{n+t+b}{n}. 
\end{align}
We treat both sides as polynomials in $b$. Our goal is to rewrite the left hand
side as a combination of $\binom{n+t_i+b}{n}$ for some $t_i$ ($i=0,\ldots, n$)
and conclude that this determines the numbers $\gamma^s(t+2, a)$ since for any 
pairwise distinct numbers $t_0, \ldots, t_n$ the polynomials $\binom{t_i+x}{n}$
are linearly independent, and $\gamma^s(t, a)$, $\varepsilon^s(t, a)$ do not 
vanish only when $t=t_i$ for some $i \in \{0, \ldots, n\}$.



\noindent \pscirclebox{4} 
Now we use the formulas (\ref{eq:bsd}) and (\ref{eq:ksd}) 
for $\dim B^{(s)}_d$ and $\dim \tilde B^{(s)}_d$ to expand the left hand 
sides of (\ref{eq:main}) and (\ref{eq:mainsp}), first calculating the sums
\begin{align*}
   \St \dim (B^{(s)}_{a-tq}) q_{b+t}  
     &= \underbrace{\St \sum_{j\geq 0} (-1)^j \dim(A^{(s)}_{a-tq-jq}) q_{b+t}}_{S_1} 
     + \underbrace{\St \sum_{j\geq 0} (-1)^j \dim(M^{(s)}_{a-q-tq-jq}) q_{b+t}}_{S_2}, \\
   \St \dim(\tilde B^{(s)}_{a-tq})q_{b+t}  
     &= \underbrace{
		\St \sum_{j\geq 0} (-1)^j \dim (\tilde A^{(s)}_{a-tq-jq}) q_{b+t}
	}_{S'_1} + \underbrace{
		\St \sum_{j\geq 0} (-1)^j \dim (\tilde M^{(s)}_{a-q-tq-jq}) q_{b+t}
	}_{S'_2}. 
\end{align*}

\begin{lemma} \label{lemma:sums}
	Let $\alpha(t) = \sum_{j\geq 0} (-1)^j \binom{n+1}{j} f(t-jq)$. Then
	\[ 
		f(a+bq) = \sum_{t\in\Zz} \alpha(a - tq)\binom{n+t+b}{n}.
	\]
\end{lemma}

\begin{proof}
Expanding the right hand side gives 
$\sum_{u\in\Zz} f(a+uq) \left(
	     \sum_{i+j=b-u} (-1)^j \binom{n+1}{j} \binom{n+i}{n}
	  \right)$ 
and the nested sum is equal to the coefficient of $z^{b-u}$ in $(1-z)^{n+1} 
\cdot (1-z)^{-n-1} = 1$. So it is just $\delta_{b, u}$.
\end{proof}

\begin{lemma} \label{lemma:identities}
	The following identities hold
	\[ 
		q_{a+bq} = \sum_{t\in\Zz} \dim A^{(s)}_{a-tq} \binom{n+b+t}{n},
		\quad
	   s_{a+bq} = \sum_{t\in\Zz} \dim \tilde A^{(s)}_{a-tq} \binom{n+b+t}{n}. 
	\]
\end{lemma}

\begin{proof}
This follows immediately from Lemma \ref{lemma:sums} for $f(t) = q_t$ 
and $f(t) = s_t$ and from 
the formulas (\ref{eq:asd}), (\ref{eq:atsd}) for the dimensions of 
$A_d$ and $\tilde A_d$.
\end{proof}

\begin{lemma} \label{lemma:sums2}
	Let $\alpha(t) = \sum_{j\geq 0} (-1)^j f(t-jq)$. Then
	\[
		\sum_{t\in\Zz} \alpha(a - tq) q_{b+t} 
		= \sum_{t\in \Zz} f(a - tq) \binom{n+t+b}{n}. 
	\]
\end{lemma}

\begin{proof}
We expand the left hand side
\[
  LHS = \sum_{t\in\Zz} \sum_{j\geq 0} (-1)^j f(a - tq - jq) q_{b+t} 
      = \sum_{u\in\Zz} f(a+qu) \left(\sum_{i\leq b+u} (-1)^{b+u-i}q_i \right).
\]
and observe that $\sum_{i\leq x} (-1)^{x-i} q_i = \binom{n+x}{n}$, which yields
the result.
\end{proof}

Now by Lemma \ref{lemma:identities}, $S_1$ and $S'_1$ cancel out with 
$q_{a+bq}$ and $s_{a+bq}$ on the left hand sides of (\ref{eq:main}) and 
(\ref{eq:mainsp}) respectively, and Lemma \ref{lemma:sums2} shows that 
\[
	S_2 = \sum_{t\in \Zz} \dim M^{(s)}_{a-(t+1)q} \binom{n+t+b}{n}, 
	\quad
	S'_2 = 	\sum_{t\in \Zz} \dim \tilde M^{(s)}_{a-(t+1)q} \binom{n+t+b}{n}.
\]
Putting these into (\ref{eq:main}) and (\ref{eq:mainsp}) 
(and replacing $t$ by $t-2$) yields 
\begin{align}\label{eq:main2}
 \sum_{t\in\Zz} \left( 
	\frac{1}{\MM}\dim M^{(s)}_{a-(t-1)q} - \gamma^s(t, a) 
 \right) \binom{n+t-2+b}{n} &= 0, \\
\label{eq:mainsp2}
	\St \left(
		\frac{1}{\MM}\dim \tilde M^{(s)}_{a-(t-1)q} - \varepsilon^s(t, a) 
	\right) \binom{n+t-2+b}{n} &= 0. 
\end{align}


\noindent \pscirclebox{5} 
We want to conclude from (\ref{eq:main2}) and (\ref{eq:mainsp2}) that 
\[
	\gamma^s(t, a) = \frac{1}{\MM} \dim M^{(s)}_{a-(t-1)q}
	\quad \text{and} \quad
	\varepsilon^s(t, a) =  \frac{1}{\MM} \dim \tilde M^{(s)}_{a-(t-1)q},
\]
which with the formulas (\ref{eq:bta}) and (\ref{eq:dta}) immediately gives
\[
	\beta^s(t, a) = \dim C^{(s)}_{a-tq}
	\quad \text{and} \quad
	\delta^s(t, a) = \dim \tilde C^{(s)}_{a-tq}.	
\]

Observe that by the formula (\ref{eq:bta}), $\gamma^s(t+1, a)\neq 0$ implies 
$B^{(s)}_{a-tq} \neq 0$. Note that $B^{(s)}_d \neq 0$ only 
for $0\leq d\leq (q-1)(n+1)$ and $K^{(s)}_d\neq 0$ only for $1\leq d\leq (q-1)(n+1)+1$
(since $D^{(s)}_d\neq 0$ if and only if $0\leq d\leq (q-1)(n+1)$). 
Therefore if
$\frac{1}{\MM} \dim M^{(s)}_{a-(t-1)q} - \gamma^s(t, a)$ is non-zero, then $0 \leq 
a - (t-1)q \leq (n+1)(q-1)$. This can happen for at most $n+1$ values of $t$,
so (\ref{eq:main2}) is an equation of linear dependence of the polynomials 
$\binom{t_i+x}{n}$ for $n+1$ distinct values $t_i$ (similarly with 
(\ref{eq:mainsp2})). As they are clearly linearly
independent (by the Vandermonde determinant), we conclude that all coefficients
are zero. This yields

\begin{theorem} \label{theorem:fofs}
	The coefficients $\beta^s(t, a)$ and $\gamma^s(t, a)$ (resp.
	$\delta^s(t, a)$ and $\varepsilon^s(t, a)$) of $\Oo(t)$ and 
	$\Ss(t)$ in $\F^s_*(\Oo(a))$ (resp. $\F^s_*(\Ss(a))$) 
	and are given by the formulas
	\begin{align*}
		\beta^s(t, a)  &= \dim C^{(s)}_{a-tq}, &
		\gamma^s(t, a) &=  \frac{1}{\MM} \dim M^{(s)}_{a-(t-1)q}.\\
		  \delta^s(t, a)      &= \dim \tilde C^{(s)}_{a-tq}, &
		  \varepsilon^s(t, a) &= \frac{1}{\MM}\dim \tilde M^{(s)}_{a-(t-1)q}. 
	\end{align*}
\end{theorem}

\begin{remark}
Since $h^1(S(t)) = 0$ and $h^1(S\otimes S(t)) = \delta_{t, 0}$ for $n$ odd
and $2\cdot \delta_{t,0}$ for $n$ even (\cite{AL}, Lemma 2.3), by the 
projection formula ((\ref{eq:proj1}) with $\Ff=\Oo(d)$, $\Gg=\Ss$ and $i=1$)
we obtain
\[
	\dim M^{(s)}_d = 2^{\lceil n/2 \rceil} h^1(\F^{s*}\Ss(d-q))
	\quad \text{and} \quad
	\dim \tilde M^{(s)}_d = 2^{\lceil n/2 \rceil} h^1(\Ss\otimes \F^{s*}\Ss(d-q)).
\]
\end{remark}

\section{Vanishing and non-vanishing}
\label{section:van}


\subsection{Symmetry} 
\label{subsection:symmetry}


For smooth complete varieties $X$, $Y$ and a proper morphism $f:X\to Y$,
the relative Serre duality (\cite{RD}) can be expressed in the following
form (e.g. \cite{FM}, 3.4, formula 3.20):
\[
	Rf_* D(\Ee) = D(Rf_* \Ee), 
\]
where $D(\Ee) = \Ee^\dual\otimes \omega$. Now since the Frobenius morphism is
an affine morphism, the higher direct images vanish, and we get

\begin{proposition*}
	Let $X$ be a smooth projective variety over an algebraically closed 
	field $k$ of characteristic $p>0$ and let $\F:X\to X$ be the 
	absolute Frobenius morphism. Then for
	any vector bundle $\Ee$ on $X$ we have
	\[ 
	  \F_* (\Ee^\dual\otimes \omega_X) = (\F_* \Ee)^\dual \otimes \omega_X. 
	\]
\end{proposition*}

On a smooth $n$-dimensional quadric $Q_n$, we have $\omega_{Q_n} = 
\Oo_{Q_n}(-n)$ and $\Ss^\dual = \Ss(1)$. This shows that, in the notation 
of Section \ref{section:fofs},
\begin{align*}
	\beta^s(t, a)       &= \beta^s(-t-n, -a-n), &
	\delta^s(t, a)      &= \delta^s(-t-n, -a+1-n), \\
	\gamma^s(t, a)      &= \gamma^s(-t+1-n, -a-n), &
	\varepsilon^s(t, a) &= \varepsilon^s(-t+1-n, -a+1-n). 
\end{align*} 
Setting $t = 0$ and using Theorem \ref{theorem:fofs} we deduce

\begin{proposition} \label{prop:symmetry}
\begin{align*}
	C^{(s)}_d &= C^{(s)}_{n(q-1)-d},   &  
	\tilde C^{(s)}_d &= \tilde C^{(s)}_{n(q-1)+1-d}, \\
	M^{(s)}_d &= M^{(s)}_{n(q-1)+q-d}, &  
	\tilde M^{(s)}_d &= \tilde M^{(s)}_{n(q-1)+q+1-d}. 
\end{align*}
\end{proposition}
We also need the symmetry of $A^{(s)}$ and $\tilde A^{(s)}$:

\begin{proposition} \label{prop:symmetrya}
\begin{align*}
	A^{(s)}_d &= A^{(s)}_{n(q-1)+q-d},  &  
	\tilde A^{(s)}_d &= \tilde A^{(s)}_{n(q-1)+q+1-d}.
\end{align*}
\end{proposition}

\begin{proof}
Use formulas (\ref{eq:asd}) and (\ref{eq:atsd}).
\end{proof}


\subsection{Which summands appear ($p>2$)}

In this section we assume that $p > 2$. We will be able to show precisely which 
summands do appear in higher Frobenius push-forwards of ACM bundles. In the
view of Theorem \ref{theorem:fofs}, this is equivalent to
determining which graded parts of the zero-dimensional graded modules $C^{(s)}$, 
$M^{(s)}$, $\tilde C^{(s)}$ and $\tilde C^{(s)}$ treated in Section 
\ref{section:algebras} are non-zero.

For brevity, let $D=D^{(1)}=k[x_0,\ldots, x_N]/(x_0^p,\ldots, x_N^p)$.

\begin{langerlemma}[Proposition 3.1 in \cite{AL}, see also \cite{ALE}] 
	Let $0\leq e\leq p$ and let 
	$x\in D_d$ with $d\leq \frac{1}{2}(N+1)(p-1) - e$. Assume that
	$Q^e\cdot x = 0$. Then there exists a $y\in D_{d-2(p-e)}$ such
	that $x = Q^{p-e}\cdot y$. 	
\end{langerlemma}

\begin{lemma} \label{lemma:crucial}
	Let $(\Phi, \Psi)$, $\Phi, \Psi\in {\rm M}_{k\times k}(D_1)$
	be an \emph{arbitrary} matrix factorization of $Q$ over the ring $D$.
	Let $0 < e \leq p$ and let $h\in D^k_d$ with 
	$d\leq \frac{1}{2}(N+1)(p-1)-e$.
	\begin{enumerate}
	  \item If $Q^e\cdot h = 0$ then there exists $g$
	    such that $h = Q^{p-e}\cdot g$.
	  \item If $Q^{e-1}\cdot\Phi(h) = 0$ then there exists $g$
	    such that $h = Q^{p-e}\cdot\Psi(g)$.
	\end{enumerate}
\end{lemma}

\begin{proof} \ghostlf

1. This is Langer's Lemma above.

2. Let us first show that there exists $f$ such that $h = \Psi(f)$.
If $e < p$ then since $Q^e\cdot h = \Psi(Q^{e-1}\cdot\Phi(h))= 0$, by
(1.) there exists $f'$ such that $h = Q^{p-e}\cdot f'
= \Psi(Q^{p-e-1}\cdot\Phi(f'))$. So we take $f = Q^{p-e-1}\cdot\Phi(f')$.
Assume that $e = p$. Applying (1.) to $\Phi(h)$
and $e = p-1$ gives us $u$ such that $\Phi(h) = Q\cdot u$. 
Therefore $\Phi(h - \Psi(u)) = 0$. Now because of what we have just proven
for $e = 1$ there exists $v$ such that $h - \Psi(u) = \Psi(v)$, so
we can put $f = u + v$.

To finish the proof, we observe that since $h = \Psi(f)$, we have
$0 = Q^{e-1}\Psi(h) = Q^e\cdot f$. So again by (1.) there
exists $g$ such that $f = Q^{p-e}\cdot g$ and hence $h = Q^{p-e}\cdot\Psi(g)$.
\end{proof}

\begin{proposition} \label{lemma:mvanish} \ghostlf
	\begin{enumerate}
	  \item $M^{(1)}_d = 0$ for $d \leq \dn$ or $d \geq \dn + p$.
	  \item $\tilde M^{(1)}_d = 0$ for $d \leq \dn$ or $d > \dn + p$.
	\end{enumerate}
\end{proposition}

\begin{proof}
By Proposition 
\ref{prop:symmetry} it is sufficient to show the vanishings for $d\leq \dn$.

1. See the proof of Proposition 3.4 from \cite{AL}.

2. We mimic the proof of the aforementioned Proposition.
We need to prove that if $g_0$ is a vector of homogeneous polynomials
of degree $\leq \dn - 1$ such that
\[ 
	x_0^p\cdot g_0 = \Phi(h) + \sum_{i=1}^{N} x_i^p\cdot g_i \eqno (*) 
\]
then there exist $h'$, $h_i$, $i=0, \ldots, N$ such that 
$g_0 = x_0^p\cdot h_0 + \sum_{i=1}^N x_i^p\cdot h_i + \Psi(h')$.

By (*) and the previous lemma, there exist $h', h_0, h'_1, \ldots, h'_n$ 
such that $h = Q^{p-1}\Phi(h') + x_0^p\cdot h_0 + \sum_{i=1}^N x_i^p\cdot h'_i$.
Putting this back into (*) yields
\begin{align*}
 g_0\cdot x_0^p &= Q^p\cdot h' + x_0^p\cdot \Psi(h_0) 
	      + \sum_{i=1}^N x_i^p \cdot (\Psi(h'_i) + g_i) \\
           &= x_0^{2p}\cdot h' + (Q - x_0^2)^p\cdot h' + x_0^p \cdot \Psi(h_0) 
	      + \sum_{i=1}^N x_i^p\cdot(\Psi(h'_i) + g_i). 
\end{align*}
Hence $x_0^p \cdot (g_0 - x_0^p\cdot h' - \Psi(h_0)) 
= \sum_{i=1}^N x_i^p \cdot h''_i$ for some $h''_i$. But $x_0^p$ is not a zero 
divisor in $k[x_0, \ldots, x_N]/(x_1^p, \ldots, x_N^p)$, which shows that 
$g_0 - x_0^p\cdot h' - \Psi(h_0) = \sum_{i=1}^N x_i^p\cdot h_i$ 
for some $h_i$.
\end{proof}

\begin{proposition} \label{prop:cvan} \ghostlf
	\begin{enumerate}
	  \item $C^{(1)}_d \neq 0$ if and only if $0\leq d \leq n(p-1)$.
 	  \item $\tilde C^{(1)}_d \neq 0$ if and only if $1\leq d \leq n(p-1)$.
	\end{enumerate}
\end{proposition}

\begin{proof}
Since $\dim C^{(1)}_0 = 1$, $\dim C^{(1)}_{-1}=0$ and $\dim C^{(1)}_d 
= \dim C^{(1)}_{n(p-1)-d}$,
it suffices to check that $\dim C^{(1)}_d$ is increasing for $d\leq \dn$. But by
the previous lemma and the exact sequence (\ref{eq:exact1})
\[
	\dim C^{(1)}_d = \dim B^{(1)}_d = \sum_{i\geq 0} (-1)^i \dim A^{(1)}_{d-pi}.
\]
Now the formula (\ref{eq:asd}) yields the result. The proof for 
$\tilde C^{(1)}$ is analogous.
\end{proof}

\begin{proposition} \label{prop:mvan} \ghostlf
	\begin{enumerate}
	  \item $M^{(1)}_d \neq 0$ if and only if $\dn < d < \dn + p$,
	  \item $\tilde M^{(1)}_d \neq 0$ if and only if 
		$\dn < d \leq \dn + p$.
	\end{enumerate}
\end{proposition}

\begin{proof}
The exact sequences (\ref{eq:exact1}) and (\ref{eq:exact2}) together with 
Proposition \ref{lemma:mvanish} yield
\[
	\dim M^{(1)}_d = \sum_{i\in\Zz} (-1)^i \dim A^{(1)}_{d+pi} 
\]
for $d \in (\dn, \dn + p]$ and $M^{(1)}_d=0$ otherwise. The same is true for 
$\tilde M^{(1)}$ and $\tilde A^{(1)}$ in place of $M^{(1)}$ and $A^{(1)}$.

Let $D(d, N) = \sum_{j=0}^N (-1)^j \binom{N}{j} \binom{n+d - pj}{n}$
and $E(d, N) = \sum_{i\in\Zz} (-1)^i D(d + ip, N)$. Then  
by formulas (\ref{eq:asd}) and (\ref{eq:atsd}) 
$\dim A^{(1)}_d = D(d, N) + D(d-1, N)$ and $\dim \tilde A^{(1)}_d = \MM D(d-1, N)$,
so, in the view of the above formulas for $\dim M^{(1)}_d$ and $\dim \tilde M^{(1)}_d$,
we want to prove that for $p$ odd, $E(d-1, N)$ is
always non-zero and that $E(d, N) + E(d-1, N) = 0$ if and only if $p$ 
divides $d-\dn$. 

We proceed by induction on $N$, proving also that $E(d, N)$ is increasing
with respect to $d$ for $d \in (\dn, \frac{1}{2}N(p-1)]$.
For $N = 1$ we have $D(d, 1) = 1$ for 
$d = 0, \ldots, p-1$ and $0$ otherwise, so $E(d, 1) \neq 0$ for all $d$ and 
$E(d, 1) = -E(d-1, 1)$ if and only if $p$ divides $d$. 

For the induction step, we use the formula $E(d, N) = \sum_{j=0}^{p-1} E(d-j, N-1)$, 
the fact that $E(d, N-1)>0$ for $d\in (\frac{1}{2}(n-1)(p-1), \frac{1}{2}(n-1)(p-1) + p]$
and $E(d, N) + E(d-1, N) > 0$ for $d\in (\frac{1}{2}(n-1)(p-1), \frac{1}{2}(n-1)(p-1) + p)$
(being the dimension of a vector space) and the symmetry for $M^{(1)}$ and 
$\tilde M^{(1)}$.
\end{proof}

\begin{theorem} \label{theorem:van} 
	Let $p>2$, $s\geq 1$ and $n>2$. Then
	\begin{enumerate} 
	  \item $\F^s_* (\Oo(a))$ contains $\Oo(t)$ if and only if 
		$0\leq a - tq \leq n(q-1)$, 
	  \item $\F^s_* (\Oo(a))$ contains $\Ss(t)$ if and only if 
		\[
			\left(\dn - p + 1\right)q/p  
			\leq a - tq \leq 
			\left(\dn - 1\right)q/p + n(q/p - 1), 
		\]
	\item $\F^s_* (\Ss(a))$ contains $\Oo(t)$ if and only if 
		$1\leq a - tq \leq n(q-1)$, 
	\item $\F^s_* (\Ss(a))$ contains $\Ss(t)$ if and only if 
		\[
			\left(\dn - p + 1\right)q/p + 1 - \delta_{s, 1} 
			\leq a - tq \leq 
			\left(\dn - 1\right)q/p + n(q/p - 1)  + \delta_{s,1}. 
		\]
	\end{enumerate}
\end{theorem}

\begin{proof}

Denote the upper and lower bounds in 1. -- 4. by $\beta^s_0$, $\beta^s_1$, 
\ldots, $\varepsilon^s_0$ and $\varepsilon^s_1$. By Propositions \ref{prop:cvan} 
and \ref{prop:mvan} together with Theorem \ref{theorem:fofs} 
we obtain the required assertion for $s=1$. Observe that
\[
	\beta_0^s \leq \delta_0^s \leq \gamma_0^s \leq \varepsilon_0^s \leq 
	\gamma_1^s \leq \varepsilon_1^s \leq \beta_1^s = \delta_1^s. 
\]

1. $\F^s_* \Oo(a)$ contains $\Oo(t)$ if and only if either there exists 
an $i$ such that $\F^{s-1}_* (\Oo(a))$ contains $\Oo(i)$ and $\F_* (\Oo(i))$ contains
$\Oo(t)$, or there exists an $i$ such that $\F^{s-1}_* (\Oo(a))$ contains $\Ss(i)$ and 
$\F_* (\Ss(i))$ contains $\Oo(t)$. By the induction assumption, this holds if and
only if there exists an integer $i$ such that either
\[ 
	\beta_0^{s-1} \leq a - iq/p \leq \beta_1^{s-1} 
	\quad \text{and} \quad 
	\beta^1_0 \leq i-tp \leq \beta^1_1 
	\eqno (*) 
\]
or 
\[
	\gamma_0^{s-1} \leq a - iq/p \leq \gamma_1^{s-1} 
	\quad \text{and}\quad 
	\delta^1_0 \leq i-tp \leq \delta^1_1. 
	\eqno (**) 
\]

We have the following simple observation: \emph{if $A$, $B$, $C$, $D$, $a$, 
$t$, $p$, $q'$ are integers satisfying $B-A \geq q' > 0$, $D-C>0$, then there 
exists an integer $i$ such that
\[
	A \leq a - iq' \leq B \quad\text{and}\quad C\leq i-tp \leq D 
\]
if and only if $Cq' + A \leq a - tpq' \leq Dq' + B$ (and the ,,only if'' part 
remains true if we omit the assumption that $B-A \geq q'$).}

Using this observation with $(A, B, C, D)=(\beta_0^{s-1}, \beta_1^{s-1}, 
\beta^1_0, \beta^1_1)$ and $q' = q/p$, we see that $(*)$ is equivalent to
$\beta^s_0 \leq a - tq \leq \beta^s_1$. Again with $(A, B, C, D) = 
(\gamma_0^{s-1}, \gamma_1^{s-1}, \delta^1_0, \delta^1_1)$ this shows that
$(**)$ \emph{implies} $q/p \delta^1_0 + \gamma_0^{s-1} \leq a - tq \leq q/p 
\delta^1_1 + \gamma_1^{s-1}$.
Now because the first interval contains the second one, we see that 
$\F^s_* \Oo(a)$ contains $\Oo(t)$ if and only if $ \beta^s_0 \leq a - tq 
\leq \beta^s_1$.

2. Analogously, $\F^s_* \Oo(a)$ contains $\Ss(t)$ if and only if there 
exists an $i$ such that either
$
	\gamma_0^{s-1} \leq a - iq/p \leq \gamma_1^{s-1} 
	\quad \text{and}\quad
	\varepsilon^1_0 \leq i-tp \leq \varepsilon^1_1 
$
or 
$
	\beta_0^{s-1} \leq a - iq/p \leq \beta_1^{s-1} 
	\quad \text{and}\quad 
	\gamma^1_0 \leq i-tp \leq \gamma^1_1 
$.
Using the observation from (1.), we see that this happens if and only if 
$q'\varepsilon^1_0 + \beta_0^{s-1} \leq a - tq \leq q' \gamma^1_1+\beta_1^{s-1}$ 
and these bounds are equal to $\gamma^s_0$ and $\gamma^s_1$.

The proofs of (3.) and (4.) are similar. 
\end{proof}


\subsection{Which summands appear ($p=2$)}

In this section we investigate the case when $p = 2$. As before, we first deal 
with the case $s = 1$. Let us first establish the following version of
Langer's lemma used in the preceding section.

\begin{lemma} \label{lemma:relevant} 
	Let $\chara(k)=2$, $N\geq 0$. 
	Let $M_N$ be the set of all monomials in $k[x_0, \ldots, x_N]$
	not in $I:=(x_0^2, \ldots, x_N^2)$ which
	contain at least one variable each monomial of $Q$ (except for possibly 
	$x_0^2$), 
	but are not divisible by any monomial of $Q$.
	Then $M_N$ forms a basis of $(I:(Q))/(I+(Q))$. 
\end{lemma}
\begin{proof}
The proof is by induction on $N$, starting with $N \leq 0$, for which 
$Q \in I$ and the statement is obvious.

Induction step: Renaming the last two variables, we have 
$Q = xy + Q'$. Take
$f\in (I:(Q))$ and write 
$f \equiv f_{00} + x f_{10} + y f_{01} + xy f_{11}$,
$f_{\alpha\beta}\in k[x_0, \ldots, x_{N-2}]$, so
$
	0 \equiv (xy + Q')
		(\sum_{\alpha,\beta} x^\alpha y^\beta f_{\alpha\beta})
$;
comparing coefficients in $x$ and $y$ yields the 
equations $f_{00} + f_{11} Q' \equiv 0$ and $f_{\alpha\beta} Q'\equiv 0$
for $(\alpha,\beta)\neq(1,1)$ (modulo $(x_0^2, \ldots, x_{N-2}^2)$).
By the induction assumption, 
$f_{10} = g_{10} Q' + r_{10}$ and $f_{01} = g_{01} Q' + r_{01}$, 
where $r_{\alpha\beta}$ is a unique linear combination of elements of $M_{N-2}$ of 
appropriate degree. We then have
\begin{align*}
f &= f_{00} 
     + x(g_{10} Q' + r_{10}) + y(g_{01} Q' + r_{01})
     + xy f_{11} \\
  &= Q(f_{11} + x g_{10} + y g_{01}) + x r_{10} + y r_{01}. 
\end{align*}
But $M_N = xM_{N-2}\cup y M_{N-2}$,
so we see that $M_N$ spans the quotient in question. 

For linear independence, let us write
$
	Q \cdot g \equiv \sum_{m\in M_N} a_m m 
$
with $a_m\in k$ and $g\in k[x_0,\ldots, x_N]$. Then for any 
monomial $x_i x_{i+1}$ of $Q$, monomials divisible by $x_i x_{i+1}$ do not 
occur on the 
left-hand side, so $g$ is in the ideal spanned by the variables $x_i$ and $x_{i+1}$;
in other words, every term of $g$ has at least one variable from each
term of $Q$ (except possibly $x_0^2$). 
But that means that $Q\cdot g = 0$, forcing the
combination to be trivial in $k[x_0,\ldots, x_N]/(x_0^2, \ldots, x_N^2)$; but
$M_N$ is clearly linearly independent in this ring.
\end{proof}

\begin{corollary} 
	$\gamma^1(t, a) = 1$ if $a-2(t-1) = \lfloor\frac{n}{2}\rfloor + 1$ or 
	if $a-2(t-1) =\lfloor\frac{n}{2}\rfloor + 2$ and $n$ is odd, and 
	$\gamma^1(t, a) = 0$ otherwise.
\end{corollary}

\begin{proof}
By the exact sequence
\[ 0 \to D^{(1)}[-2]/(0:Q) \xto{Q} D^{(1)} \to B^{(1)} \to 0 \]
we have $\dim B_{d} = \dim D_{d} + \dim M'_{d-2} - \dim B_{d-2}$,
where $M'=(I:(Q))/(I+(Q))$, so 
\begin{equation}\label{eq:bsd2}
 \dim B_{d} = \sum_{j\geq 0} (-1)^j \dim D_{d-2j} + \sum_{j\geq 0} (-1)^j \dim M'_{d-2(j+1)}.
\end{equation}
Proceeding exactly as in Section \ref{section:fofs}, but replacing the use of \ref{eq:bsd}
by \ref{eq:bsd2} gives $\gamma^1(t, a) = \frac{1}{\MM} \dim M'_{a-2(t-1)}$, which
together with Lemma \ref{lemma:relevant} yields the result.
\end{proof}

Now we shall prove an analogue of Lemma \ref{lemma:crucial}:

\begin{lemma}
	Let $\chara\, k = 2$, $n>0$ and let $\phi_n$ and $\psi_n$ be the 
	matrices defined in Section \ref{section:spinor}. Let $h$ be a vector 
	with polynomial entries of length $\sizephin$. Suppose that all entries 
	of $h$ are homogeneous polynomials of degree $d$.
	\begin{enumerate}
	\item If $Q_n\cdot h\in (x_0^2, \ldots, x_{n+1}^2)$ and $d\leq \lfloor 
	  n/2\rfloor$, then there exists a vector $g$ with polynomial 
	  entries for which $h\equiv Q_n\cdot g$ modulo 
	  $(x_0^2, \ldots, x_{n+1}^2)$.
	\item If $Q_n\cdot\phi_n(h)\in (x_0^2, \ldots, x_{n+1}^2)$ and 
	  $d\leq \lceil n/2\rceil - 1$, then there exists a vector $g$ with
	  polynomial entries for which $h\equiv \psi_n(g)$ modulo 
	  $(x_0^2, \ldots, x_{n+1}^2)$ (the same is true with $\phi_n$ and 
	  $\psi_n$ exchanged).
	\item If $\phi_n(h)\in (x_0^2, \ldots, x_{n+1}^2)$ and 
	  $d\leq \lceil n/2\rceil$, then there exists a vector $g$ with 
	  polynomial entries for which $h\equiv Q_n\cdot\psi_n(g)$ modulo 
	  $(x_0^2, \ldots, x_{n+1}^2)$ (the same is true with $\phi_n$ and 
	  $\psi_n$ exchanged). 
	\end{enumerate}
\end{lemma}

\begin{proof}
We work in the ring $D_n = k[x_0, \ldots, x_{n+1}]/(x_0^2, \ldots, x_{n+1}^2)$ 
and proceed by a induction on $n$.
For brevity let $\phi = \phi_n$, $\psi = \psi_n$, $\phi' = \phi_{n-2}$,
$\psi' = \phi_{n-2}$, $x = x_n$, $y = x_{n+1}$, $Q = Q_n$ and $Q' = Q_{n-2}$.

1. This follows from Lemma \ref{lemma:relevant} above.

2. Let us divide $h$ in two pieces: $h = (h_0, h_1)$. We can write 
$h_i$, $i=0, 1$ as 
$
	h_i = h_i^{00} + x h_i^{01} + y h_i^{10} + xy h_i^{11} 
$
where $h_i^{jk}$ are polynomials in $x_0, \ldots, x_{n-1}$. 

Using the recurrence relations 
\[
	\phi = \left( \begin{array}{cc}
		\phi' & x \cdot id  \\
		y\cdot id & \psi'  \\
	\end{array} \right), \quad
	\psi = \left( \begin{array}{cc}
		\psi' & x \cdot id  \\
		y\cdot id & \phi'  \\
	\end{array} \right), \quad
	Q = xy + Q',
\]
our assumption on $h$ takes the form
$(xy + Q')(\phi'(h_0) + x h_1) = 0$, 
$(xy+Q')(\psi'(h_1) + y h_0) = 0$.
By comparing coefficients in $x$ and $y$ we see that
\begin{align} 
	\label{eq:row1}
	   Q' \phi'(h_0^{00}) &= 0,           & 
	   Q' \psi'(h_1^{00}) &= 0, \\
	\label{eq:row2}
	   Q' (\phi'(h_0^{01}) + h_1^{00}) &= 0  &
	   Q'(\psi'(h_1^{10}) + h_0^{00})&= 0,    \\
	\label{eq:row3}
	   Q' \phi'(h_0^{10}) &= 0,  & 
	   Q' \psi'(h_1^{01}) &= 0,  \\
	\label{eq:row4} 
	   \phi'(h_0^{00}) + Q'\phi'(h_0^{11}) + Q' h_1^{10} &= 0, &
	   \psi'(h_1^{00}) + Q'\psi'(h_1^{11}) + Q' h_0^{01} &= 0. 
\end{align}

By (\ref{eq:row3}) and the induction assumption, there exist $g_0^{10}$ and 
$g_1^{01}$ such that
$
	h_0^{10} = \psi'(g_0^{10})
$ and $
	h_1^{01} = \phi'(g_1^{01})
$.
Observe also that by (\ref{eq:row2}) and (1.), there exist $g_0^{01}$ and
$g_1^{10}$ such that
$
	\phi'(h_0^{01}) + h_1^{00} = Q' g_0^{01}
$ and $
	\psi'(h_1^{10}) + h_0^{00} = Q' g_1^{10}
$.
Putting this into (\ref{eq:row4}) and using the induction assumption once again 
gives us $g_0^{11}$ and $g_1^{11}$ such that
$
	g_1^{10} + h_0^{11} = \psi'(g_0^{11}) 
$ and $
	g_0^{01} + h_1^{11} = \phi'(g_1^{11})
$.
Finally define $g_0^{00} = \phi'(g_1^{10}) + h_1^{10}$ and 
$g_1^{00} = \psi'(g_0^{01}) + h_0^{01}$ and observe that $g = (g_0, g_1)$ 
defined by $g_i = g_i^{00} + x g_i^{01} + y g_i^{10} + xy g_i^{11}$
satisfies $\phi(g) = h$.

3.
Let us first prove that there exists an $f$ such that $h = \psi(f)$. 
Decomposing $h$ as before, we have
\begin{align}
	\label{eq:rrow1}
	   \phi'(h_0^{00}) &= 0,     & 
	   \psi'(h_1^{00}) &= 0, \\
	\label{eq:rrow2}
	   \phi'(h_0^{01}) + h_1^{00} &= 0,  &
	   \psi'(h_1^{10}) + h_0^{00} &= 0,    \\
	\label{eq:rrow3}
	   \phi'(h_0^{10}) &= 0,     &
	   \psi'(h_1^{01}) &= 0, \\
	\label{eq:rrow4}
	  \phi'(h_0^{11})+ h_1^{10} &= 0,     & 
	  \psi'(h_1^{11}) + h_0^{01} &= 0. 
\end{align}
By (\ref{eq:rrow3}) and the induction assumption, there exist $f_0^{10}$ and 
$f_1^{01}$ such that $h_0^{10} = \psi'(f_0^{10})$ and $h_1^{01} = 
\phi'(f_1^{01})$. Observe also that 
\[ 
	Q'\cdot\phi'(h_0^{11}) = Q'\cdot h_1^{10}
	= \phi'(\psi'(h_1^{10})) = \phi'(h_0^{00}) = 0, 
\]
and similarly $Q'\cdot\psi'(h_1^{11}) = 0$, therefore by (1.) there exist 
$f_0^{11}$ and $f_1^{11}$ such that $h_0^{11} = \psi'(f_0^{11})$ and $h_1^{11} 
= \phi'(f_1^{11})$. Finally set $f_0^{00} = h_1^{10}$, $f_1^{00} = h_0^{01}$, 
$f_0^{01} = 0$ and $f_1^{10} = 0$ and observe that $f = (f_0, f_1)$, 
$f_i = f_i^{00} + x f_i^{01} + y f_i^{10} + xy f_i^{11}$ satisfies $h=\psi(f)$.

Now since $Q \cdot f = \phi(\psi(f)) = \phi(h) = 0$, by (2.) there exists a 
$g$ such that $f = Q\cdot g$, therefore $h = \psi(f) = Q\cdot\psi(g)$.
\end{proof}

Proceeding exactly as in Propositions \ref{prop:cvan} 
and \ref{prop:mvan} and Theorem \ref{theorem:van}, one obtains

\begin{theorem} \label{theorem:van2} \ghostlf
	\begin{enumerate}
	  \item $F^s_* (\Oo(a))$ contains $\Oo(t)$ if and only if 
		$0\leq a - tq \leq n(q-1)$, 
	  \item $F^s_* (\Oo(a))$ contains $S(t)$ if and only if 
		\[ 
			\left(\lfloor\frac{n}{2}\rfloor - 1\right)\frac{q}{2}
			\leq a - tq \leq 
			n(q-1) - q - \left(\lfloor\frac{n}{2}\rfloor - 1\right)\frac{q}{2}, 
		\]
	  \item $F^s_* (\Ss(a))$ contains $\Oo(t)$ if and only if 
		$1\leq a - tq \leq n(q-1)$, 
	  \item $F^s_* (\Ss(a))$ contains $\Ss(t)$ if and only if 
		\[ 
			\left(\lfloor\frac{n}{2}\rfloor - 1\right)\frac{q}{2} + 1 + \delta_{s,1}\cdot\delta_{n,odd}
			\leq a - tq \leq 
			n(q-1) - q - \left(\lfloor\frac{n}{2}\rfloor - 1\right)\frac{q}{2} - \delta_{s,1}\cdot\delta_{n,odd}, 
		\]
		where $\delta_{n, odd} = 1$ if $n$ is odd and $0$ otherwise.\hfill $\square$
	\end{enumerate}
\end{theorem}

\section{Corollaries}

The following simple fact follows from Theorems \ref{theorem:van} and 
\ref{theorem:van2}.

\begin{corollary}
	For any ACM bundle $\Ee$ on $Q_n$, there are only finitely many $t\in\Zz$
	for which there exists an $s$ such that
	$\Oo(t)$ or $\Ss(t)$ appears in $\F^s_* \Ee$.
\end{corollary}

Now we proceed to extend the main results from \cite{AL}. 

\begin{definition} 
	A coherent sheaf $\Ff$ on a variety $X$ is called 
	\emph{quasi-exceptional} if $\Ext^i(\Ff, \Ff) = 0$ for $i>0$. 
	$\Ff$ is \emph{tilting} if it is quasi-exceptional, 
	Karoubian generates the bounded derived category $D^b(X)$ and the 
	algebra $\Hom_X(\Ff, \Ff)$ has finite global dimension.
\end{definition}

\begin{lemma} 
	We have $\Ext^1(\Ss(a), \Ss(a+1))\neq 0$ and $\Ss(a)$ is quasi-exceptional.
\end{lemma}

\begin{proof}
For the first statement, tensor the sequence (\ref{eq:exacts}) by $\Ss(a)$ and
write the long cohomology exact sequence. The second statement follows even
simpler from (\ref{eq:exacts}).
\end{proof}

The following theorem extends slightly the main Theorem 1.1 from \cite{AL}.

\begin{theorem} \label{theorem:tilting} 
	Let $n>2$. Then $\F^s_* \Oo_{Q_n}$ is tilting if and only if 
	one of the following holds:
\begin{enumerate}
	\item $s = 1$ and $p > n$,
	\item $s = 2$, $n = 4$ and $p = 2, 3$,
	\item $s\geq 2$, $n$ is odd and $p\geq n$.
\end{enumerate}
\end{theorem}

\begin{proof}
If $p>2$, this is Theorem 1.1 from \cite{AL} (and can also be easily deduced 
from Theorem \ref{theorem:van}). Thus the only new part here is to show that in 
the case $p=2$, $\F^s_* \Oo$ is not tilting, except for the case $s=2$
and $n=4$.

By Theorem \ref{theorem:van}, we see that $\F^s_* \Oo$ contains as direct
summands only the line bundles
\[
	\Oo, \Oo(-1), \ldots, \Oo(-\lfloor n - \frac{n}{q}\rfloor),
\]
so if $n > q$ then $\F^s_* \Oo$ does not generate the derived category.

We also see that the $\F^s_* \Oo$ contains $\Ss(t)$ for 
$\gamma_0^s\leq -tq \leq \gamma_1^s$ with $\gamma_1^s - \gamma_0^s = 
\frac{nq}{2} - n \geq 2q$ for $q\geq n \geq 6$, so in this case $\F^s_* \Oo$ 
contains two consecutive twists of $\Ss$, therefore is not quasi-exceptional
by the above lemma.

Finally we work out the cases $n=3, 4, 5$ by hand: for $n=3$, $\F^2_*\Oo$
contains $\Ss$ and $\Ss(-1)$; for $n=4$, $\F^3_*\Oo$ contains $\Ss(-1)$ and 
$\Ss(-2)$; for $n=5$, $\F^2_*\Oo$ contains $\Ss(-1)$ and $\Ss(-2)$, so
they (and the higher push-forwards) are not quasi-exceptional. 
For $n=3,4,5$ and $s=1$, we have $n>q$. It remains to check the case $n=4$, $s=2$:
$\F^2_* \Oo$ contains $\Ss(-1)$ and $\Oo(-i)$ for $i=0,1,2,3$, so it is 
tilting.
\end{proof}

\subsection*{A note on singular quadrics}

It would be interesting to extend the above results to singular quadrics.
It should be noted first that the ring $S/(Q)$ with $Q$ a quadratic form
not of full rank is no longer of finite Cohen-Macaulay type. Recently,
N. Addington in \cite{A} constructed the so-called \emph{spinor sheaves},
which are analogues of spinor bundles. Among them, there are always one
or two (depending on the parity of the rank of $Q$)
\emph{maximal} spinor sheaves (i.e., coming from a maximal linear
subspace on the quadric) 
and they have nearly the same cohomological properties as the spinor bundles.
In particular, if we denote by $\Ss$ the maximal spinor sheaf of the sum
of the two and assume that $\F_* (\Oo(a))$ and $\F_* (\Ss(a))$ decompose
into direct sums of twists of $\Oo$ and $\Ss$, it is easy to see that
the results from Section \ref{section:fofs} hold true
almost without change (one has to replace the factors
$2^{\lfloor n/2\rfloor +1}$ by $2^{\lfloor r/2\rfloor}$, $r$ being the rank
of $Q$).

\end{document}